\documentclass[dvipdfmx,11pt]{amsart}
\usepackage{amsmath,amssymb,amsthm}
\usepackage{mathtools,stmaryrd}
\usepackage[utf8]{inputenc}

\theoremstyle{plain}
\newtheorem{thm}{Theorem}[section]

\newtheorem{lem}[thm]{Lemma}
\newtheorem{cor}[thm]{Corollary}
\newtheorem{conj}[thm]{Conjecture}

\theoremstyle{definition}

\newtheorem{rem}[thm]{Remark}

\numberwithin{equation}{section}

\newcommand{\Z}{\mathbb{Z}}
\newcommand{\Q}{\mathbb{Q}}
\newcommand{\R}{\mathbb{R}}
\newcommand{\C}{\mathbb{C}}

\newcommand{\cA}{\mathcal{A}}
\newcommand{\cS}{\mathcal{S}}

\newcommand{\cZ}{\mathcal{Z}}
\newcommand{\hatA}{{\widehat{\mathcal{A}}}}
\newcommand{\hatS}{{\widehat{\mathcal{S}}}}
\newcommand{\hatU}{{\widehat{\mathcal{U}}}}
\newcommand{\frD}{\mathfrak{D}}

\newcommand{\bbra}[1]{\llbracket #1 \rrbracket}

\newcommand{\bk}{{\boldsymbol{k}}}
\newcommand{\bl}{{\boldsymbol{l}}}
\newcommand{\bp}{{\boldsymbol{p}}}
\newcommand{\bs}{{\boldsymbol{s}}}

\DeclareMathOperator{\Ker}{Ker}

\title{On the refined Kaneko--Zagier conjecture for general integer indices}
\author{Masataka Ono \and Shuji Yamamoto}
\date{\today}

\thanks{The first author is supported by Waseda University Grants for Special Research Projects (Grant No. 2021C-673). The second author is supported by JSPS KAKENHI JP18H05233, JP18K03221 and JP21K03185. }

\address[M.~Ono]{
	Global Education Center\\
	Waseda University\\
	1-6-1, Nishi-Waseda, Shinjuku-ku, Tokyo, 169-8050\\
	Japan}
\email{m-ono@aoni.waseda.jp}
\address[S.~Yamamoto]{
    Department of Mathematics, Faculty of Science and Technology, Keio University, 
    3-14-1 Hiyoshi, Kouhoku-ku, Yokohama 223-8522, Japan
}
\email{yamashu@math.keio.ac.jp}

\begin{document}

\begin{abstract}
The refined Kaneko--Zagier conjecture claims that the algebras spanned by two kinds of 
``completed'' finite multiple zeta values, called $\hatA$- and $\hatS$-MZVs, are isomorphic. 
Recently, Komori defined $\hatS$-MZVs of general integer (i.e., not necessarily positive) indices, 
extending the existing definition for positive indices. 
In view of the refined Kaneko--Zagier conjecture, Komori's work suggests that these extended values are 
closely connected to $\hatA$-MZVs of general indices, which can be defined in an obvious way. 
In this paper, we show that the generalization of the refined Kaneko--Zagier conjecture for general integer indices 
is actually deduced from the conjecture for positive indices. 
The key ingredient is an inductive formula for $\hatA$-MZVs  or $\hatS$-MZVs of indices 
which contain at least one non-positive entry. 
%We deduce the refined Kaneko--Zagier conjecture for $\hatA$- and $\hatS$-multiple zeta values of general index from the original one. The key ingredient is the evaluation of unified multiple zeta functions at non-positive integers.
\end{abstract}

\maketitle

\section{Introduction}\label{sec:Intro}

For any $r \in \Z_{\ge0}$, we call an $r$-tuple $\bk=(k_1, \dots, k_r)$ of integers an \emph{index}. 
Such $\bk$ is said \emph{positive} if all entries $k_i$ are positive. 
In order to emphasize that a given index is not necessarily positive, we often call it a \emph{general} index. 
The length $r$ of an index $\bk=(k_1,\ldots,k_r)$ is called the \emph{depth} of $\bk$ and, 
if $\bk$ is positive, the sum of its entries $k_1+\cdots+k_r$ is called the \emph{weight} of $\bk$. 
Let $\varnothing$ denote the unique index of depth $0$, called the \emph{empty index}. 
We regard $\varnothing$ as a positive index of weight $0$. 

A positive index $\bk=(k_1, \dots, k_r)$ is said \emph{admissible} if $\bk=\varnothing$ or $k_r\ge2$.
For an admissible positive index $\bk=(k_1, \dots, k_r)$, the multiple zeta value (MZV) $\zeta(\bk)$ is a real number defined by
\[\zeta(\bk)
\coloneqq \sum_{0<n_1<\dots<n_r}\frac{1}{n^{k_1}_1 \dots n^{k_r}_r}.\]
MZVs have been extensively investigated from various viewpoints, 
and several variants of them are also considered. 
In particular, since the discovery of a surprising conjecture by Kaneko and Zagier \cite{KZ}, 
two kinds of finite multiple zeta values, which we call $\cA$- and $\cS$-MZVs, are actively studied. 
In this paper, we treat the completed versions of them, called $\hatA$-MZVs \cite{R2,S} and $\hatS$-MZVs \cite{J,OSY}. 
We recall the definitions of them below. 

Let $\bk=(k_1,\ldots,k_r)$ be a positive index, not necessarily admissible. 
Then the first variant, the $\hatA$-MZV $\zeta_{\hatA}(\bk)$, is an element of the $\Q$-algebra 
\[\hatA\coloneqq\varprojlim_n
\Biggl[\Biggl(\prod_{p:\text{prime}} \Z/p^n\Z\Biggr)\Bigg/\Biggl(\bigoplus_{p:\text{prime}} \Z/p^n\Z\Biggr)\Biggr], \]
in which $\Q$ is diagonally embedded. 
We define, for each prime number $p$, an element of the localization $\Z_{(p)}$ of $\Z$ with respect to 
the prime ideal $(p)$ by 
\begin{equation}\label{eq:zeta_p}
\zeta^{}_{p}(\bk) 
\coloneqq\sum_{0<n_1<\cdots<n_r<p}\frac{1}{n^{k_1}_1\cdots n^{k_r}_r} \in \Z_{(p)}. 
\end{equation}
Then $\zeta_{\hatA}(\bk)$ is defined by 
\begin{equation}\label{eq:zeta_hatA}
\zeta^{}_{\hatA}(\bk)
\coloneqq\Bigl(\bigl(\zeta^{}_{p}(\bk) \bmod{p^n}\bigr)_p\Bigr)_n \in \hatA. 
\end{equation}
Next, let $\cZ$ denote the $\Q$-subalgebra of $\R$ generated by all MZVs 
and set $\overline{\cZ}\coloneqq\cZ/\pi^2\cZ$ (note that $\pi^2=6\zeta(2)\in\cZ$). 
Then the second variant of MZV, the $\hatS$-MZV $\zeta^{}_{\hatS}(\bk)$, is 
a formal power series in an indeterminate $t$ with coefficients in $\overline{\cZ}$ defined by 
\begin{align}
\zeta^{}_{\hatS}(\bk)
&\coloneqq
\sum_{i=0}^{r}(-1)^{k_{i+1}+\cdots+k_r}\zeta^{*}(k_1, \dots, k_i)\label{eq:zeta_hatS}\\
&\quad
\times\sum_{l_{i+1}, \dots, l_r \ge0}
\Biggl\{\prod_{j=i+1}^{r}\binom{k_j+l_j-1}{l_j}t^{l_j}\Biggr\}
\zeta^{*}(k_r+l_r, \dots, k_{i+1}+l_{i+1}) \bmod{\pi^2}. \notag
\end{align}
Here, $\zeta^{*}(\bk) \in \cZ$ denotes the harmonic regularized MZV, which is defined for any positive index $\bk$ 
(for the definition of the regularized MZVs, refer to \cite{IKZ}). 
This notion of the $\hatS$-MZV is equivalent to the $\Lambda$-adjoint MZVs introduced by Jarossay \cite{J}. 

It is conjectured that the $\hatA$-MZVs and the $\hatS$-MZVs satisfy exactly the same algebraic relations 
(the refined Kaneko--Zagier conjecture \cite{J, OSY, R2}; see \S\ref{sec:KZconj} for more details). 
Though this conjecture seems still far from resolution, 
it suggests that some phenomenon observed in $\hatA$-MZVs should have an analogue in $\hatS$-MZVs and vice versa. 

Now let us observe that the definition of $\zeta^{}_\hatA(\bk)$ given by \eqref{eq:zeta_p} and \eqref{eq:zeta_hatA} 
is meaningful for general indices $\bk$ as well. 
On the other hand, the definition \eqref{eq:zeta_hatS} of $\zeta^{}_\hatS(\bk)$ does not work for such $\bk$, 
since the regularized MZVs are defined only for positive indices. 
This gap was filled by Komori \cite{Ko}, who found a natural definition of $\zeta^{}_\hatS(\bk)$ for general $\bk$. 
In fact, he introduced a complex analytic function $\zeta_\hatU(s_1,\ldots,s_r;t)$, 
called the unified multiple zeta function. 
This function is entire with respect to $s_1,\ldots,s_r$ and satisfies the interpolation property 
\[\zeta^{}_\hatS(\bk)\equiv\zeta^{}_\hatU(\bk;t)\bmod \pi i\cZ[\pi i]\bbra{t} \]
for any positive index $\bk$. 
Then Komori defined $\zeta^{}_\hatS(\bk)$ for general $\bk$ to be $\zeta^{}_\hatU(\bk;t)$
modulo $\pi i\cZ[\pi i]\bbra{t}$. 

Given these extensions of $\hatA$- and $\hatS$-MZVs for general indices, 
it is natural to ask whether the refined Kaneko--Zagier conjecture is generalized to these extended values. 
In this paper, we will prove that this generalization can be actually deduced from the original one 
(the precise statement is given in Theorem \ref{thm:general KZ}): 

\begin{thm}\label{thm:Intro1}
If the refined Kaneko--Zagier conjecture concerning the $\hatA$- and $\hatS$-MZVs of positive indices is true, 
then the same is true for those values of general indices. 
\end{thm}

Indeed, this result is an immediate consequence of the following: 

\begin{thm}\label{thm:Intro2}
For any index $\bk\in\Z^r$, there is a collection $(c_{\bk,\bl}(x))_\bl$ of rational polynomials 
$c_{\bk,\bl}(x)\in\Q[x]$ indexed by a finite number of positive indices $\bl$ such that 
\[\zeta^{}_\hatA(\bk)=\sum_{\bl}c_{\bk,\bl}(\bp)\zeta^{}_\hatA(\bl)\quad 
\text{and}\quad \zeta^{}_\hatS(\bk)=\sum_{\bl}c_{\bk,\bl}(t)\zeta^{}_\hatS(\bl)\]
hold simultaneously, where we set $\bp\coloneqq\bigl((p \bmod{p^n})_{p}\bigr)_{n}\in\hatA$. 
Moreover, the depth (resp.\ the weight) of these positive indices  $\bl$ are at most 
the number (resp.\ the sum) of positive entries of $\bk$. 
\end{thm}

For example, we have 
\begin{align*}
\zeta^{}_{\hatA}(k,-1)
&=\frac{1}{2}\Bigl((\bp^2-\bp)\zeta^{}_{\hatA}(k)-\zeta^{}_{\hatA}(k-1)-\zeta^{}_{\hatA}(k-2)\Bigr),\\
\zeta^{}_{\hatS}(k,-1)
&=\frac{1}{2}\Bigl((t^2-t)\zeta^{}_{\hatS}(k)-\zeta^{}_{\hatS}(k-1)-\zeta^{}_{\hatS}(k-2)\Bigr)
\end{align*}
for any integer $k\ge 3$.  

The proof of Theorem \ref{thm:Intro2} is based on a formula which reduces the depth $r$ 
whenever $\bk$ contains a non-positive entry. 
First we state this formula for $\zeta^{}_\hatA(\bk)$. 
Let $B_n$ denote the Seki-Bernoulli number defined by $ze^z/(e^z-1)=\sum_{n=0}^{\infty}B_n z^n/n!$ and $\delta_{xy}$ denote the Kronecker delta. 

\begin{thm}\label{thm:mainA}
Let $\bk=(k_1, \dots, k_r)$ be an index of depth $r\ge 2$, 
and suppose that $k_i=-k$ for some $i\in\{1,\ldots,r\}$ and $k\ge 0$. 
\begin{enumerate}
\item If $i=1$, we have
\begin{align*}
\zeta^{}_{\hatA}(\bk)=
\frac{1}{k+1}\sum_{j=0}^{k+1}\binom{k+1}{j}B_j\Bigl((-1)^j\zeta^{}_{\hatA}(k_2-k-1+j,\ldots,k_r)&\\
-\delta_{j,k+1}\,\zeta^{}_{\hatA}(k_2,\ldots,k_r)&\Bigr).
\end{align*}
\item If $1<i<r$, we have
\begin{align*}
\zeta^{}_{\hatA}(\bk)=
\frac{1}{k+1}\sum_{j=0}^{k+1}\binom{k+1}{j}B_j\Bigl((-1)^j\zeta^{}_{\hatA}&(k_1,\dots,k_{i+1}-k-1+j,\dots,k_r)\\
\qquad-\zeta^{}_{\hatA}&(k_1,\dots,k_{i-1}-k-1+j,\dots,k_r)\Bigr). 
\end{align*}
\item If $i=r$, we have
\begin{align*}
\zeta^{}_{\hatA}(\bk)=
\frac{1}{k+1}\sum_{j=0}^{k+1}\binom{k+1}{j}B_j\Bigl((-1)^j\zeta^{}_{\hatA}&(k_1,\dots,k_{r-1}) \bp^{k+1-j}\\
-\zeta^{}_{\hatA}&(k_1,\dots,k_{r-1}-k-1+j)\Bigr). 
\end{align*}
\end{enumerate}
\end{thm}

The corresponding formula for $\zeta^{}_\hatS(\bk)$ holds, in fact, at the level of the complex analytic function. 
See \S\ref{sec:hatU} for the definition of the unified multiple zeta function $\zeta^{}_\hatU(\bs;t_+,t_-)$. 

\begin{thm}\label{thm:mainU}
Let $\bs=(s_1, \dots, s_r)$ be an $r$-tuple of complex numbers with $r\ge 2$, 
and suppose that $s_i=-k$ for some $i\in\{1,\ldots,r\}$ and an integer $k\ge 0$. 
\begin{enumerate}
\item If $i=1$, we have
\begin{align*}
\zeta_\hatU(\bs;t_+,t_-)=
\frac{1}{k+1}\sum_{j=0}^{k+1}\binom{k+1}{j}B_j\Bigl((-1)^j\zeta_\hatU&(s_2-k-1+j,\dots,s_r;t_+,t_-)\\
-\zeta_\hatU&(s_2,\dots,s_r;t_+,t_-)\,t_+^{k+1-j}\Bigr). 
\end{align*}
\item If $1<i<r$, we have
\begin{align*}
\zeta_\hatU(\bs;t_+,t_-)=
\frac{1}{k+1}\sum_{j=0}^{k+1}\binom{k+1}{j}B_j\Bigl((-1)^j\zeta_\hatU&(s_1,\dots,s_{i+1}-k-1+j,\dots,s_r;t_+,t_-)\\
-\zeta_\hatU&(s_1,\dots,s_{i-1}-k-1+j,\dots,s_r;t_+,t_-)\Bigr). 
\end{align*}
\item If $i=r$, we have
\begin{align*}
\zeta_\hatU(\bs;t_+,t_-)=
\frac{1}{k+1}\sum_{j=0}^{k+1}\binom{k+1}{j}B_j\Bigl((-1)^j\zeta_\hatU&(s_1,\dots,s_{r-1};t_+,t_-)\,t_-^{k+1-j}\\
-\zeta_\hatU&(s_1,\dots,s_{r-1}-k-1+j;t_+,t_-)\Bigr). 
\end{align*}
\end{enumerate}
\end{thm}

From Theorem \ref{thm:mainU}, we obtain the following: 
\begin{cor}\label{cor:Z[pi i]}
For any $r \in \Z_{\ge0}$ and $\bk \in \Z^r$, 
we have $\zeta^{}_{\hatU}(\bk; t_+, t_-) \in \cZ[\pi i]\bbra{t_+, t_-}$.
\end{cor}

The contents of this paper is as follows. 
In \S\ref{sec:hatA}, we prove Theorem \ref{thm:mainA} by using Faulhaber's formula on power sums. 
In \S\ref{sec:hatU}, we prove Theorem \ref{thm:mainU} along the same line as in the previous section, 
by considering (a slight modification of) the order on non-zero integers due to Kontsevich. 
In \S\ref{sec:KZconj}, we explain the Kaneko--Zagier conjecture and its refinement, 
and present the precise version of Theorem \ref{thm:Intro1}.
Finally, in \S\ref{sec:values}, we discuss the values of $\zeta^{}_{\hatU}(\bs; t_+, t_-)$ at non-positive integer points as an application of our results.

\section{Main theorem for $\zeta_\hatA$}\label{sec:hatA}

In this section, we prove Theorem \ref{thm:mainA}. 

\begin{lem} \label{lem:FaulhaberA}
For any non-negative integers $k$ and $a<b$, we have
\begin{equation}\label{eq:FaulhaberA}
\sum_{a<n<b}n^k=\frac{1}{k+1}\sum_{j=0}^{k+1}\binom{k+1}{j}B_j\bigl\{(-1)^{j}b^{k+1-j}-a^{k+1-j}\bigr\}.
\end{equation}
Here we use the convention that $a^{k+1-j}=1$ for $a=0$ and $j=k+1$. 
\end{lem}

\begin{proof}
From Faulhaber's formula 
\[\sum_{0<n\le b}n^k=\frac{1}{k+1}\sum_{j=0}^{k}\binom{k+1}{j}B_j\cdot b^{k+1-j},\]
we see that 
\begin{align*}
\sum_{a<n<b}n^k
&=-b^k+\frac{1}{k+1}\sum_{j=0}^{k}\binom{k+1}{j}B_j\bigl\{b^{k+1-j}-a^{k+1-j}\bigr\}\\
&=-b^k+\frac{1}{k+1}\sum_{j=0}^{k+1}\binom{k+1}{j}B_j\bigl\{b^{k+1-j}-a^{k+1-j}\bigr\}. 
\end{align*}
Then the desired expression follows, since $B_1=1/2$ and $B_j=0$ for any odd integer $j\ge3$. 
\end{proof}

\begin{proof}[Proof of Theorem \ref{thm:mainA}]
By substituting \eqref{eq:FaulhaberA} into the definition of $\zeta_p(\bk)$, 
one sees that $\zeta_p(\bk)$ satisfies the formula of the same form as given in Theorem \ref{thm:mainA}: 
For example, 
\begin{align*}
\zeta_p(k_1,-k,k_3)&=\sum_{0<n_1<n_2<n_3<p}\frac{n_2^k}{n_1^{k_1}n_3^{k_3}}\\
&=\sum_{0<n_1<n_3<p}\frac{1}{n_1^{k_1}n_3^{k_3}}\cdot
\frac{1}{k+1}\sum_{j=0}^{k+1}\binom{k+1}{j}B_j\bigl\{(-1)^j n_3^{k+1-j}-n_1^{k+1-j}\bigr\}\\
&=\frac{1}{k+1}\sum_{j=0}^{k+1}\binom{k+1}{j}B_j\bigl\{(-1)^j\zeta_p(k_1,k_3-k-1+j)-\zeta_p(k_1-k-1+j,k_3)\bigr\}. 
\end{align*}
Then Theorem \ref{thm:mainA} follows immediately. 
\end{proof}

\begin{rem}\label{rem:hatA a,b}
Given integers $a_p,b_p$ with $0\le a_p\le b_p\le p$ for each prime $p$, 
we may define a generalization of $\zeta^{}_\hatA(\bk)$ by 
\[\zeta^{}_{\hatA}(\bk; a_\bp, b_\bp)\coloneqq
\Biggl(\biggr(\sum_{a_p<n_1<\dots<n_r<b_p}\frac{1}{n^{k_1}_1\dots n^{k_r}_r}\bmod{p^n}\biggr)_p\Biggr)_n,\]
where $a_\bp$ and $b_\bp$ stand for $((a_p\bmod p^n)_p)_n,((b_p\bmod p^n)_p)_n\in\hatA$, respectively. 
Theorem \ref{thm:mainA} has a natural generalization to these values $\zeta^{}_{\hatA}(\bk; a_\bp, b_\bp)$, 
and the similarity to Theorem \ref{thm:mainU} becomes more apparent. 
For example, the formula for $i=1$ is
\begin{align*}
\zeta^{}_{\hatA}(\bk; a_\bp, b_\bp)=
\frac{1}{k+1}\sum_{j=0}^{k+1}\binom{k+1}{j}B_j\Bigl((-1)^j\zeta^{}_{\hatA}(k_2-k-1+j,\ldots,k_r; a_\bp, b_\bp)&\\
-\zeta^{}_{\hatA}(k_2,\ldots,k_r; a_\bp, b_\bp)\,a_\bp^{k+1-j}&\Bigr).  
\end{align*}
\end{rem}

\section{Main theorem for $\zeta_\hatU$}\label{sec:hatU}

In this section, we prove Theorem \ref{thm:mainU}. 
First we recall the definition of the unified multiple zeta function introduced by Komori \cite{Ko}. 
For $t\in\C\setminus[1,\infty)$, we define a version of multiple zeta function of Hurwitz type by 
\begin{equation}\label{eq:Hurwitz type}
\zeta(s_1,\ldots,s_r;t)\coloneqq\sum_{0<n_1<\cdots<n_r}
\frac{1}{(n_1-t)^{s_1}\cdots(n_r-t)^{s_r}}.  
\end{equation}
Here the power $(n-t)^s$ is defined by using the principal value of the logarithm $\log(n-t)$. 
This series \eqref{eq:Hurwitz type} converges absolutely and uniformly on compact sets of the region 
\[\Re(s_j)+\cdots+\Re(s_r)>r+1-j \quad (j=1,\ldots,r), \]
and meromorphically continued to $(s_1,\ldots,s_r)\in\C^r$ \cite[Theorems 3 and 4]{Ma}. 
For $t_+\in\C\setminus(-\infty,-1]$ and $t_-\in\C\setminus[1,\infty)$, we define 
\begin{equation}\label{eq:zeta_hatU}
\zeta^{}_\hatU(s_1,\ldots,s_r;t_+,t_-)\coloneqq \sum_{i=0}^r(-1)^{s_{i+1}+\cdots+s_r}
\zeta(s_1,\ldots,s_i;-t_+)\,\zeta(s_r,\ldots,s_{i+1};t_-), 
\end{equation}
where $(-1)^s=e^{\pi is}$ for any $s\in\C$ (note that we have slightly changed the definition from that  
given in \cite[p.~225]{Ko}; Komori's $(t_1,t_2)$ is our $(-t_+,t_-)$). 
The function $\zeta^{}_\hatU(s_1,\ldots,s_r;t)$ mentioned in \S\ref{sec:Intro} is defined as 
\[\zeta^{}_\hatU(s_1,\ldots,s_r;t)\coloneqq \zeta^{}_\hatU(s_1,\ldots,s_r;0,t). \]

In the following, we often abbreviate the pair of $t_+$ and $t_-$ as $t_\pm$.  
In \cite{Ko}, it is proved that $\zeta^{}_\hatU(\bs;t_\pm)$ is extended to a holomorphic function on 
\[(s_1,\ldots,s_r,t_+,t_-)\in \C^r\times(\C\setminus\Z_{\le 1})\times(\C\setminus\Z_{\ge 1}).\] 
In particular, $\zeta^{}_\hatU(\bs;t_\pm)$ is an entire function on $\C^r$ for fixed $t_\pm$. 
%To ensure the convergence of all series appearing in \eqref{eq:zeta_hatU}, 
%however, one has to assume that $\Re(s_i)>1$ for $i=1,\ldots,r$. 

The definition \eqref{eq:zeta_hatU} of $\zeta^{}_\hatU(\bs;t_\pm)$ can be rewritten 
by using the so-called \emph{Kontsevich order}. 
Fixing $t_\pm$, we define sets $I_+$, $I_-$ and $I$ by 
\[I_+\coloneqq\{n+t_+ \mid n \in \Z_{\ge0}\}, \quad 
I_{-}\coloneqq\{-n+t_- \mid n \in \Z_{\ge0}\}, \quad I\coloneqq I_+ \amalg I_-,\]
and introduce a total order $\prec$ on the set $I$ by 
\[t_+\prec 1+t_+\prec 2+t_+\prec\cdots\prec n+t_+\prec\cdots
\prec -n+t_-\prec\cdots\prec -2+t_-\prec -1+t_-\prec t_-.\]
Then, if $\Re(s_i)>1$ for all $i$, we may write \eqref{eq:zeta_hatU} as 
\[\zeta^{}_{\hatU}(s_1, \dots, s_r; t_\pm)
=\sum_{\substack{a_1,\ldots,a_r\in I\\ t_+ \prec a_1 \prec \dots \prec a_r \prec t_-}}
\frac{1}{a_1^{s_1} \cdots a_r^{s_r}}.\]
Here, for $a=-n+t_- \in I_-$, we understand $1/a^{s}=(-1)^s/(n-t_-)^s$.

\begin{lem}
For $a,b\in I$ with $a\prec b$, set
\[F_a^b(s)=\sum_{a\prec c\prec b}\frac{1}{c^s}.\]
Then $F_a^b(s)$ can be continued to an entire function on $\C$.
\end{lem}
\begin{proof}
Let $\zeta^{}_{H}(s, a)$ denote the Hurwitz zeta function defined by 
\[\zeta^{}_H(s,a)\coloneqq\sum_{n=0}^\infty \frac{1}{(n+a)^s}.\] 
Note that it satisfies $\zeta_H(s,a)=\zeta(s;1-a)$, where the right hand side is defined by 
\eqref{eq:Hurwitz type} with $r=1$. 
It is well-known that $\zeta^{}_{H}(s, a)$ is meromorphically continued to the whole complex $s$-plane, 
and holomorphic except the pole at $s=1$ with the principal part $1/(s-1)$. 

If $a,b\in I_+$ or $a,b\in I_-$, then $F_a^b(s)$ is a finite sum and the statement is obvious. 
On the other hand, if $a=n_a+t_+\in I_+$ and $b=-n_b+t_-\in I_-$, 
then we have 
\[F^{a}_{b}(s)=\sum_{n>n_a}\frac{1}{(n+t_+)^s}+\sum_{n>n_b}\frac{(-1)^s}{(n-t_-)^s}
=\zeta_H(s,n_a+1+t_+)+(-1)^s\zeta_H(s,n_b+1-t_-).  \]
Hence this can be meromorphically continued to $\C$, and the singularity at $s=1$ is removed.  
\end{proof}

By the definition of $F^{b}_{a}(s)$, for $s_1, \dots, s_r \in \C$ satisfying $\Re(s_1),\ldots,\Re(s_r)>1$, 
we see that 
\begin{equation} \label{eq:zetahatU i}
\zeta^{}_{\hatU}(s_1, \dots, s_r; t_\pm)
=\sum_{t_+\prec a_1\prec\cdots\prec a_{i-1}\prec a_{i+1}\prec\cdots\prec a_r\prec t_-}
\frac{F_{a_{i-1}}^{a_{i+1}}(s_i)}{a_1^{s_1}\cdots a_{i-1}^{s_{i-1}} a_{i+1}^{s_{i+1}}\cdots a_r^{s_r}}. 
\end{equation}

\begin{lem}\label{lem:zetahatU i k}
Let $C$ be a positive real number. The series expression \eqref{eq:zetahatU i} holds 
for $s_1, \dots, s_r \in \C$ satisfying $\Re(s_i)>-C$ and $\Re(s_j)>C+2$ $(\forall j \neq i)$. 
In particular, if $k \in \Z_{\ge 0}$ and $s_j \in \C$ $(j \neq i)$ satisfying $\Re(s_j)>k+2$, 
we have
\begin{equation}\label{eq:zetahatU i k}
\zeta^{}_{\hatU}(s_1, \dots, s_r; t_\pm)|_{s_i=-k}
=\sum_{t_+\prec a_1\prec\cdots\prec a_{i-1}\prec a_{i+1}\prec\cdots\prec a_r\prec t_-}
\frac{F_{a_{i-1}}^{a_{i+1}}(-k)}{a_1^{s_1}\cdots a_{i-1}^{s_{i-1}} a_{i+1}^{s_{i+1}}\cdots a_r^{s_r}}. 
\end{equation}
\end{lem}
\begin{proof}
Let us write $\frD_C$ for the domain of $(s_1,\ldots,s_r)\in\C^r$ 
defined by $\Re(s_i)>-C$ and $\Re(s_j)>C+2$ $(\forall j \neq i)$. 
For the first asssertion, 
it suffices to show that the expression in \eqref{eq:zetahatU i} gives a holomorphic function on $\frD_C$ 
(recall that $\zeta^{}_\hatU$ is known to be holomorphic on $\C^r$). 

We divide the series in the right hand side of \eqref{eq:zetahatU i} as 
\[\sum_{(*)}=\sum_{(*),\,a_{i-1},a_{i+1}\in I_+}+\sum_{(*),\,a_{i-1},a_{i+1}\in I_-}
+\sum_{(*),\,a_{i-1}\in I_+, a_{i+1}\in I_-},\]
where $(*)$ stands for the condition 
$t_+\prec a_1\prec\cdots\prec a_{i-1}\prec a_{i+1}\prec\cdots\prec a_r\prec t_-$. 
Then the first sum is equal to 
\[\sum_{j=i+1}^r(-1)^{s_{j+1}+\cdots+s_r}\zeta(s_1,\dots,s_j;-t_+)\,\zeta(s_r,\dots,s_{j+1}; t_-),\]
and all series $\zeta(s_1,\dots,s_j;-t_+)$ and $\zeta(s_r,\dots,s_{j+1}; t_-)$ appearing here 
converge uniformly on compact sets of $\frD_C$. Hence the first sum gives a holomorphic function on $\frD_C$. 
The same argument works for the second sum as well. 
For the third sum, note that 
\[F_{a_{i-1}}^{a_{i+1}}(s_i)=\zeta(s_i;-t_+)-\sum_{n=1}^{n_{i-1}}\frac{1}{(n+t_+)^{s_i}}
+(-1)^{s_i}\zeta(s_i;t_-)-\sum_{n=1}^{n_{i+1}}\frac{1}{(-n+t_-)^{s_i}}\]
holds for $a_{i-1}=n_{i-1}+t_+\in I_+$ and $a_{i+1}=-n_{i+1}+t_-\in I_-$. 
Hence the third sum is equal to 
\begin{align*}
&\bigl\{\zeta(s_i; -t_+)+(-1)^{s_i}\zeta(s_i; t_-)\bigr\}\cdot
(-1)^{s_{i+1}+\cdots+s_r}\zeta(s_1,\dots,s_{i-1}; -t_+)\zeta(s_r,\dots,s_{i+1}; t_-)\\
&-\sum_{\substack{0<n_1<\cdots<n_{i-1}\\ 0<n_i\le n_{i-1}}}
\frac{1}{(n_1+t_+)^{s_1}\cdots (n_i+t_+)^{s_i}}
\cdot (-1)^{s_{i+1}+\cdots+s_r}\zeta(s_r,\dots,s_{i+1}; t_-)\\
&-\zeta(s_1,\dots,s_{i-1}; -t_+)\cdot (-1)^{s_i+\cdots+s_r}
\sum_{\substack{0<n_r<\cdots<n_{i+1}\\ 0<n_i\le n_{i+1}}}
\frac{1}{(n_r-t_-)^{s_r}\cdots (n_i-t_-)^{s_i}}.
\end{align*}
Two nested series converge absolutely and uniformly on compact sets in $\frD_C$; 
this can be seen, for example, by observing that 
\begin{multline*}
\sum_{\substack{0<n_1<\cdots<n_{i-1}\\ 0<n_i\le n_{i-1}}}\frac{1}{(n_1+t_+)^{s_1}\cdots (n_i+t_+)^{s_i}}\\
=\sum_{j=1}^{i-1}\bigl(\zeta(s_1,\ldots,s_{j-1},s_i,s_j,\ldots,s_{i-1};-t_+)
+\zeta(s_1,\ldots,s_j+s_i,\ldots,s_{i-1};-t_+)\bigr). 
\end{multline*}
Therefore, the third sum also gives a holomorphic function on $\frD_C$, 
and the proof of our first assertion is complete. 
The second assertion is deduced from the first, applied to all $C>k$.
\end{proof}

\begin{lem} \label{lem:FaulhaberU}
For a non-negative integer $k$ and $a, b \in I$ with $a\prec b$, we have
\[F_a^b(-k)=\frac{1}{k+1}\sum_{j=0}^{k+1}\binom{k+1}{j}B_j\bigl((-1)^j b^{k+1-j}-a^{k+1-j}\bigr).\]
\end{lem}

\begin{proof}
For $a,b \in I_+$ or $a,b \in I_-$, the statement follows immediately from Faulhaber's formula. 
If $a\in I_+$ and $b\in I_-$, we have
\[F_a^b(s)=\zeta(s;-a)+(-1)^s\zeta(s;b).\]
Let $B_n(x)\in\Q[x]$ be the $n$-th Bernoulli polynomial defined by 
$ze^{xz}/(e^z-1)=\sum_{n=0}B_n(x)z^n/n!$ 
(for the properties of Bernoulli polynomials, we refer to \cite[\S 4.3]{AIK}). 
Since $\zeta(s;a)=\zeta_H(s,1-a)$ and
\[\zeta_H(-k,a)=-\frac{B_{k+1}(a)}{k+1},\]
we have
\[\zeta(-k;a)=-\frac{B_{k+1}(1-a)}{k+1}=(-1)^k\frac{B_{k+1}(a)}{k+1}
=\frac{(-1)^k}{k+1}\sum_{j=0}^{k+1}(-1)^j\binom{k+1}{j}B_j\,a^{k+1-j}.\]
Thus we obtain
\begin{align*}
F_a^b(-k)&=\zeta(-k;-a)+(-1)^k\zeta(-k;b)\\
&=\frac{1}{k+1}\sum_{j=0}^{k+1}\binom{k+1}{j}B_j\bigl((-1)^jb^{k+1-j}-a^{k+1-j}\bigr). \qedhere
\end{align*}
\end{proof}

\begin{proof}[Proof of Theorem \ref{thm:mainU}]
By using Lemma \ref{lem:zetahatU i k} and Lemma \ref{lem:FaulhaberU}, 
this is proved in the same way as Theorem \ref{thm:mainA}. 
\end{proof}

\begin{proof}[Proof of Theorem \ref{thm:Intro2}]
We proceed by induction on the depth $r$ of $\bk=(k_1,\ldots,k_r)$. 
The statement is trivial for $r=0$, since we have 
$\zeta^{}_\hatA(\varnothing)=1$ and $\zeta^{}_\hatS(\varnothing)=1$ by convention. 
Let us consider the case $r=1$. If $\bk=(k)$ with $k>0$, there is nothing to prove.  
On the other hand, if $\bk=(-k)$ with $k\ge 0$, we have 
\[\zeta^{}_\hatA(-k)
=\frac{1}{k+1}\sum_{j=0}^{k+1}\binom{k+1}{j}B_j\bigl((-1)^j\bp^{k+1-j}-\delta_{j,k+1}\bigr)\]
and 
\[\zeta^{}_\hatS(-k)
=\frac{1}{k+1}\sum_{j=0}^{k+1}\binom{k+1}{j}B_j\bigl((-1)^j t^{k+1-j}-\delta_{j,k+1}\bigr), \]
hence the statement is true. 

Let $\bk=(k_1,\ldots,k_r)\in\Z^r$ with $r\ge 2$. 
Again, there is nothing to do if $\bk\in\Z_{>0}^r$. 
Otherwise, we can apply Theorem \ref{thm:mainA} and Theorem \ref{thm:mainU} 
to see that $\zeta^{}_\hatA(\bk)$ and $\zeta^{}_\hatS(\bk)$ admit a common expression 
as a $\Q$-linear combination of the values of depth $r-1$. 
Thus, by the induction hypothesis, the statement is also true for $\bk$. 
This completes the proof of Theorem \ref{thm:Intro2}. 
\end{proof}

\section{Kaneko--Zagier conjecture}\label{sec:KZconj}
In this section, we present the precise statement of Theorem \ref{thm:Intro1} related with the Kaneko--Zagier conjecture 
and its refinement. 

First let us recall the original conjecture on $\cA$-MZVs and $\cS$-MZVs proposed by Kaneko and Zagier \cite{KZ}. 
For a general index $\bk \in \Z^{r}$ ($r\ge0$), let 
\[\zeta^{}_{\cA}(\bk)\coloneqq\bigl(\zeta_p(\bk)\bmod p\bigr)_p
\in\cA\coloneqq\Biggl(\prod_p\Z/p\Z\Biggr)\Biggm/\Biggl(\bigoplus_p\Z/p\Z\Biggr)\]
be the $\cA$-MZV and let $\cZ_{\cA}$ denote the $\Q$-subalgebra of $\cA$ generated by $\cA$-MZVs $\zeta^{}_{\cA}(\bk)$ 
for all \emph{positive} indices $\bk \in \Z^{r}_{>0}$. 
By construction, we have the identity $\zeta^{}_\cA(\bk)=\zeta^{}_\hatA(\bk)\bmod\bp$ 
under the natural isomorphism $\cA\cong\hatA/\bp\hatA$. 

Following Komori \cite{Ko}, we define the $\cS$-MZV (resp.~$\hatS$-MZV) for a general index $\bk\in\Z^r$ by 
\[\zeta^{}_{\cS}(\bk)\coloneqq \zeta^{}_{\hatU}(\bk; 0, 0) \bmod \pi i\cZ[\pi i] \quad 
\bigl(\text{resp.}\ \zeta^{}_{\hatS}(\bk)\coloneqq \zeta^{}_{\hatU}(\bk; 0, t) \bmod \pi i\cZ[\pi i]\bbra{t}\bigr). \]
By Corollary \ref{cor:Z[pi i]}, they belong to the rings 
$\cZ[\pi i]/\pi i\cZ[\pi i]\cong\overline{\cZ}$ and $\overline{\cZ}\bbra{t}$, respectively. 

\begin{conj}[Kaneko--Zagier conjecture] \label{conj:KZ}
There exists a (necessarily unique) $\Q$-algebra isomorphism $\cZ_{\cA} \to \overline{\cZ}$ 
which sends $\zeta^{}_{\cA}(\bk)$ to $\zeta^{}_{\cS}(\bk)$ for any positive index $\bk$. 
In particular, $\cA$-MZVs and $\cS$-MZVs of positive indices satisfy the same $\Q$-linear relations, i.e., 
a $\Q$-linear relation $\sum_\bk a_\bk\zeta^{}_\cA(\bk)=0$ among $\cA$-MZVs of positive indices holds 
if and only if the corresponding relation $\sum_\bk a_\bk\zeta^{}_\cS(\bk)=0$ among $\cS$-MZVs holds.  
\end{conj}

To state the corresponding conjecture for $\hatA$- and $\hatS$-MZVs, 
we equip the $\Q$-algebras $\hatA$ and $\overline{\cZ}\bbra{t}$ with the $\bp$-adic and $t$-adic topology, 
respectively. Then let $\cZ_\hatA$ denote the closed $\Q$-subalgebra of $\hatA$ 
generated by $\bp$ and $\hatA$-MZVs $\zeta^{}_\hatA(\bk)$ of positive indices $\bk\in\Z^r$. 
Explicitly, $\cZ_\hatA$ consists of elements of the form 
$\sum_{i=1}^\infty a_i\zeta^{}_\hatA(\bk_i)\bp^{n_i}\in\hatA$ 
with $a_i\in\Q$, positive indices $\bk_i$, 
and integers $n_i\ge 0$ satisfying $n_i\to\infty$ ($i\to\infty$). 

\begin{conj}[{Refined Kaneko--Zagier conjecture, \cite{OSY}}] \label{conj:RKZ}
There exists a (necessarily unique) topological $\Q$-algebra isomorphism $\cZ_{\hatA} \to \overline{\cZ}\bbra{t}$ 
which sends $\bp$ to $t$, and $\zeta^{}_{\hatA}(\bk)$ to $\zeta^{}_{\hatS}(\bk)$ for any positive index $\bk$. 
In particular, $\hatA$-MZVs and $\hatS$-MZVs satisfy the same $\Q$-linear relations.
\end{conj}

\begin{rem} \label{rem:hatKZ to KZ}
At present, it has not been proven that the refined Kaneko--Zagier conjecture (Conjecture \ref{conj:RKZ}) 
implies the original Kaneko--Zagier conjecture (Conjecture \ref{conj:KZ}). 
In fact, assuming the refined conjecture, the original one is equivalent to the equality 
\begin{equation}\label{eq:ideal}
    \bp\cZ_\hatA=\Ker(\cZ_\hatA\to\cZ_\cA)
\end{equation}
of ideals in $\cZ_\hatA$. This equality is, however, 
a spacial case of Rosen's ``asymptotic extension conjecture'' 
\cite[Conjecture A]{R1}, which is unsolved yet. 
\end{rem}

As an obvious consequence of Theorem \ref{thm:Intro2}, we obtain the following result. 
The part (ii) is the precise version of Theorem \ref{thm:Intro1}. 

\begin{thm}\label{thm:general KZ}
\begin{enumerate}
\item 
Suppose that Conjecture \ref{conj:KZ} is true. 
Then the isomorphism $\cZ_{\cA} \to \overline{\cZ}$ sends $\zeta^{}_{\cA}(\bk)$ to $\zeta^{}_{\cS}(\bk)$ 
for any general index $\bk$. 
In particular, $\cA$-MZVs and $\cS$-MZVs of general indices satisfy the same $\Q$-linear relations.
\item 
Suppose that Conjecture \ref{conj:RKZ} is true. 
Then the isomorphism $\cZ_{\hatA} \to \overline{\cZ}\bbra{t}$ sends 
$\zeta^{}_{\hatA}(\bk)$ to $\zeta^{}_{\hatS}(\bk)$ for any general index $\bk\in\Z^r$. 
In particular, $\hatA$-MZVs and $\hatS$-MZVs of general indices satisfy the same $\Q$-linear relations. 
\end{enumerate}
\end{thm}

\section{Values at non-positive integer points}\label{sec:values}
As another application of Theorem \ref{thm:mainU}, 
we show a formula for the values of $\zeta^{}_{\hatU}(\bs; t_\pm)$ at non-positive integer points.
This is a generalization of Komori's formula \cite[Theorem 1.17]{Ko} on 
$\zeta^{}_\hatU(-k_1,\ldots,-k_r;t)=\zeta^{}_\hatU(-k_1,\ldots,-k_r;0,t)$. 
Our proof looks quite different from Komori's argument. 

\begin{thm}\label{thm:non-positive}
For any non-negative integers $k_1,\ldots,k_r$, the value $\zeta^{}_\hatU(-k_1,\ldots,-k_r;t_\pm)$ is 
a polynomial function of rational coefficients in $t_+$ and $t_-$. 
Moreover, the generating function 
\[G_r(z_1,\ldots,z_r;t_\pm)
\coloneqq \sum_{k_1,\ldots,k_r\ge 0}\zeta^{}_\hatU(-k_1,\ldots,-k_r;t_\pm)
\frac{z_1^{k_1}}{k_1!}\cdots\frac{z_r^{k_r}}{k_r!} \]
of these values can be expressed as 
\begin{equation}\label{eq:G_r}
\begin{split}
G_r&(z_1,\ldots,z_r;t_\pm)\\
&=\sum_{i=0}^r(-1)^{r-i}e^{t_+(z_1+\cdots+z_i)+t_-(z_{i+1}+\cdots+z_r)}
\prod_{l=1}^i\frac{e^{z_l+\cdots+z_i}}{1-e^{z_l+\cdots+z_i}}
\prod_{l=i+1}^r\frac{1}{1-e^{z_{i+1}+\cdots+z_l}}. 
\end{split}
\end{equation}
\end{thm}
\begin{proof}
First we treat the case of $r=1$. Since 
\[\zeta^{}_\hatU(-k;t_\pm)=(-1)^{-k}\zeta(-k;t_-)+\zeta(-k;-t_+)
=\frac{B_{k+1}(t_-)}{k+1}+(-1)^k\frac{B_{k+1}(-t_+)}{k+1}, \]
this is indeed a rational polynomial. Moreover, we have 
\begin{align*}
\sum_{k\ge 0}\zeta^{}_\hatU(-k;t_\pm)\frac{z^k}{k!}
&=\sum_{k\ge 0}B_{k+1}(t_-)\frac{z^k}{(k+1)!}+\sum_{k\ge 0}B_{k+1}(-t_+)\frac{(-z)^k}{(k+1)!}\\
&=\biggl(\frac{e^{t_-z}}{e^z-1}-\frac{1}{z}\biggr)
+\biggl(\frac{e^{t_+z}}{e^{-z}-1}-\frac{1}{-z}\biggr)\\
&=e^{t_+z}\frac{e^z}{1-e^z}-e^{t_-z}\frac{1}{1-e^z}. 
\end{align*}
This shows the claim for $r=1$. 

For $r\ge 2$, we proceed by induction on $r$. 
By using Theorem \ref{thm:mainU}, we can deduce that $\zeta^{}_\hatU(-k_1,\ldots,-k_r;t_\pm)$ is 
a rational polynomial in $t_\pm$ from the induction hypothesis. 
Moreover, by Theorem \ref{thm:mainU} (iii), we have 
\begin{align*}
&\sum_{k_r\ge 0}\zeta^{}_\hatU(-k_1,\ldots,-k_r;t_\pm)\frac{z_r^{k_r}}{k_r!}\\
&=\sum_{k_r\ge 0}\frac{z_r^{k_r}}{(k_r+1)!}\sum_{j=0}^{k_r+1}\binom{k_r+1}{j}B_j
\bigl\{(-1)^j\zeta^{}_\hatU(-k_1,\ldots,-k_{r-1};t_\pm)t_-^{k_r+1-j}\\
&\hspace{180pt}-\zeta^{}_\hatU(-k_1,\ldots,-k_{r-1}-k_r-1+j;t_\pm)\bigr\}\\
&=\sum_{j,l\ge 0}\frac{z_r^{j+l-1}}{j!\;l!}B_j\bigl\{(-1)^j\zeta^{}_\hatU(-k_1,\ldots,-k_{r-1};t_\pm)t_-^l
-\zeta^{}_\hatU(-k_1,\ldots,-k_{r-1}-l;t_\pm)\bigr\}\\
&=\sum_{l\ge 0}\zeta^{}_\hatU(-k_1,\ldots,-k_{r-1}-l;t_\pm)\frac{z_r^l}{l!}\cdot\frac{e^{z_r}}{1-e^{z_r}}
-\zeta^{}_\hatU(-k_1,\ldots,-k_{r-1};t_\pm)\cdot\frac{e^{t_-z_r}}{1-e^{z_r}}. 
\end{align*}
Therefore, we see that 
\begin{align*}
G_r(z_1,\ldots,z_r;t_\pm)
&=\sum_{k_1,\ldots,k_{r-1},l\ge 0}
\zeta^{}_\hatU(-k_1,\ldots,-k_{r-1}-l;t_\pm)\frac{z_1^{k_1}}{k_1!}\cdots\frac{z_{r-1}^{k_{r-1}}}{k_{r-1}!}
\frac{z_r^l}{l!}\cdot\frac{e^{z_r}}{1-e^{z_r}}\\
&\quad -\sum_{k_1,\ldots,k_{r-1}\ge 0}\zeta^{}_\hatU(-k_1,\ldots,-k_{r-1};t_\pm)
\frac{z_1^{k_1}}{k_1!}\cdots\frac{z_{r-1}^{k_{r-1}}}{k_{r-1}!}\cdot\frac{e^{t_-z_r}}{1-e^{z_r}}\\
&=G_{r-1}(z_1,\ldots,z_{r-1}+z_r;t_\pm)\frac{e^{z_r}}{1-e^{z_r}}
-G_{r-1}(z_1,\ldots,z_{r-1};t_\pm)\frac{e^{t_-z_r}}{1-e^{z_r}}. 
\end{align*}
On the other hand, it is elementary to show that the right hand side of \eqref{eq:G_r} satisfies the same 
recurrence relation. Thus, by induction on $r$, we obtain the identity \eqref{eq:G_r}. 
\end{proof}

\begin{rem}\label{rem:Komori}
Theorem \ref{thm:non-positive} may be rephrased as follows: the function 
\begin{align*}
F_r&(z_1,\ldots,z_r;y_\pm)\\
&\coloneqq \sum_{i=0}^r(-1)^{r-i}e^{y_+(z_1+\cdots+z_i)+y_-(z_{i+1}+\cdots+z_r)}
\prod_{l=1}^i\frac{e^{z_l+\cdots+z_i}}{1-e^{z_l+\cdots+z_i}}
\prod_{l=i+1}^r\frac{1}{1-e^{z_{i+1}+\cdots+z_l}}, 
\end{align*}
has an expansion 
\begin{equation}\label{eq:F_r expansion}
F_r(z_1,\ldots,z_r;y_\pm)
=\sum_{k_1,\ldots,k_r\ge 0}P_r(k_1,\ldots,k_r;y_\pm)\frac{z_1^{k_1}}{k_1!}\cdots\frac{z_r^{k_r}}{k_r!}
\end{equation}
with rational polynomial coefficients $P_r(k_1,\ldots,k_r;y_\pm)\in\Q[y_\pm]$, and the identity 
\[\zeta^{}_\hatU(-k_1,\ldots,-k_r;t_\pm)=P_r(k_1,\ldots,k_r;t_\pm)\]
holds for any $k_1,\ldots,k_r\ge 0$. 
This was the logical order that Komori pursued in the case of $(y_+,y_-)=(0,y)$ (cf.~\cite[Theorems 1.1 and 1.17]{Ko}). 
In our proof of Theorem \ref{thm:non-positive}, we have avoided to show the existence of 
the expansion \eqref{eq:F_r expansion} directly from the definition of $F_r(z_1,\ldots,z_r;y_\pm)$. 
\end{rem}

\begin{rem}
In the notation of Remark \ref{rem:hatA a,b}, we also have  
\[\zeta^{}_\hatA(-k_1,\ldots,-k_r;a_\bp,b_\bp)=P_r(k_1,\ldots,k_r;a_\bp,b_\bp). \]
The proof is the same as that of Theorem \ref{thm:non-positive}. 
\end{rem}

\section*{Acknowledgments}
The authors would like to thank Professor Koji Tasaka who kindly informed us that Rosen's conjecture is needed 
to deduce Conjecture \ref{conj:RKZ} from Conjecture \ref{conj:KZ}.
The authors also deeply appreciate Professor Yasushi Komori's helpful comments, 
especially on the content of \S5.

\end{document}